\documentclass{iasr}
\renewcommand\thisnumber{0}
\renewcommand\thisyear {2008}
\renewcommand\thismonth{February}
\renewcommand\thisvolume{00}
\usepackage[dvips]{graphicx}
\usepackage{amsthm}
\usepackage{enumerate}

\begin{document}

\title{Entropy of continuous maps on quasi-metric spaces}
\author[Sayyari Y. et.~al.]{Sayyari Yamin \affil{1},
       Molaei Mohammadreza \affil{1}\comma \corrauth, Moghayer Saeed M. \affil{2}}
 \address{\affilnum{1}\ Department of Pure Mathematics, Shahid Bahonar University of Kerman, Kerman, Iran.\\
           \affilnum{2}\ Netherlands Organization for Applied Scientific Research (TNO), van Mourik Broekmanweg 6, 2628XE, Delft, The Netherlands.}
 \corraddr{Molaei Mohammadreza , Department of Pure Mathematics, Shahid Bahonar University of Kerman, Kerman, Iran.
          Email: \tt mrmolaei@uk.ac.ir}
\received{23 january~ 2015}
%\accepted{2 november~ 2009}

\newtheorem{thm}{Theorem}[section]
 \newtheorem{cor}[thm]{Corollary}
 \newtheorem{lem}[thm]{Lemma}
 \newtheorem{prop}[thm]{Proposition}
 \newtheorem{defn}[thm]{Definition}
 \newtheorem{rem}[thm]{Remark}
\newtheorem{Example}{Example}[section]

%%%%% Begin Abstract %%%%%%%%%%%
\begin{abstract}
The category of metric spaces is a subcategory of quasi-metric spaces. In this paper the notion of entropy for the continuous maps of a quasi-metric space is extended via spanning and separated sets. Moreover, two metric spaces that are associated to a given quasi-metric space are introduced and the entropy of a map of a given quasi-metric space and  the maps of its associated metric spaces are compared. It is shown that the entropy of a map when symmetric properties is included is grater or equal to the entropy in the case that the symmetric property of the space is not considered.
\end{abstract}
%%%%% end %%%%%%%%%%%

%%%%% Keywords %%%%%%%%%%%
%%%% maketitle %%%%%
\maketitle

%%%% Start %%%%%%
\section{Introduction}
\label{sec1}
Historically, the term entropy was initially originated in classical thermodynamics in 1864 by Rudolf Clausius as state function of a thermodynamic system in which entropy is described as dissipative energy use of a thermodynamic system during a change of state \cite{Mu}. In mathematics, the notion of topological entropy was first introduced by McAndrew, Adler, and Konheim \cite{Adl, Ban, Si, Wa}  as an invariant of topological conjugacy.  Shannon  introduced the notion of entropy for measurable partitions of a probability space, and he applied this concept in information theory \cite{Sha}.  For metric spaces, another approach was presented by Dinaburg and Bowen \cite{Bo, Br, Din, Pa} via separating and spanning sets that has been considered and utilized in a verity of applications \cite{Ber, Dik, Ju, Ju2, KE1, KE, Mol, Pa, Ya}. In this paper the notion of entropy of a continuous map on a quasi-metric space is defined via separating and spanning sets in quasi-metric spaces. The metric spaces that are associated to a quasi-metric space are constructed. Also a method is developed for the calculation of the entropy of maps on quasi-metric spaces using the entropy of its associated metric spaces.  In theorem $4.1$ the role of the symmetric property of the space in topological entropy is shown.

\section{Entropy by using of spanning and separated sets}
\label{sec2}
In this section we extend the concept of topological entropy for the maps on quasi-metric spaces which are continuous according to a special topology defined by quasi-metrics.
\begin{definition} \cite{Sto, St}
 Let $X$ be a set, a quasi metric  is defined as a function $e: X\times X\longrightarrow [0, +\infty)$ that satisfies the following axioms.
\begin{enumerate}[\bf (1)]
\item $e(x, y)\geq 0,$
\item $e(x, y)=0\Leftrightarrow x=y,$
\item $e(x, z)\leq e(x,y)+e(y,z).$
\end{enumerate}
for all $x,y,z\in X.$
\end{definition}
$(X,e)$ is called a quasi-metric space. So quasi-metric has all the properties of metrics except symmetry.
\begin{example}
Let $X$ be the set of real numbers and let:\\
\begin{align*}
  e(x,y)= \left\{\begin{array}{rr} y-x & y\geq x \\ 1 & \mathrm{otherwise }  \end{array}\right.
\end{align*}
  Then $e$ is a quasi-metric on $X$. \\Let us check the condition 3.\\
  $$e(x,z)=\left\{ \begin{array}{ll}z-x\leq y-x+y-z=e(x,y)+e(y,z) & if ~ x\leq z \leq y\\  z-x\leq y-x+z-y=e(x,y)+e(y,z) & if ~ x\leq y \leq z \\ 1 \leq 1+z-y=e(x,y)+e(y,z) & if ~ y\leq z \leq x \\z-x \leq 1 + z-y=e(x,y)+e(y,z) & if ~ y\leq x \leq z \\ 1 \leq y-x+1=e(x,y)+e(y,z) & if ~ z\leq x \leq y \\ 1 \leq 1+1=e(x,y)+e(y,z) & if ~ z\leq y \leq x \end{array}.\right.$$
\end{example}
If $(X,e)$ is a quasi-metric space, then we define:
 \begin{align*}
& B^r_t(p)=\{x\in X: e(p,x)<t\}, ~ \mathrm{and} ~ \overline{B^r_t(p)}=\{ x\in X:e(p,x)\leq t \}.
\end{align*}
$B^r_t(p)$ is called the open right $t$-ball centered at $p$, and $\{B^r_t(p),~ p\in X,~ t,~r\in R\}$ is a base for a topology on $X$.\\
{\bf Remark:} Let $\{ x_n \}\subset R$ be any strictly increasing sequence convergence to $x$ in the topological space generated by open right $t$-balls of the Example $1$, one can easily prove that $\{x_n, x\}$ is not compact, but if $\{ x_n \}\subset R$ is strictly decreasing sequence then $\{x_n, x\}$ is compact.\\
We also define an open left $t$-ball centered at a point $p$ by:
% \begin{align*}
$$B^l_t(p)=\{x\in X: e(x,p)<t\}.$$
We define its closure by
$\overline{B^l_t(p)}=\{ x\in X:e(x,p)\leq t \}.$
%\end{align*}
An open $t$-ball centered at a point $p$ is the set:
 \begin{align*}
& B_t(p)=\{x\in X: e(x,p)<t ~ \mathrm{and} ~ e(p,x)<t\}= B^r_t(p)\cap B^l_t(p), \\
\end{align*}
\begin{center}
and its closure is
\end{center}
\begin{align*}
 &\overline{B_t(p)}=:\{ x\in X:e(x,p)\leq t ~
\mathrm{and}
~ e(p,x)\leq t\} =\overline{B^r_t(p)} \cap \overline{B^l_t(p)} .
\end{align*}
\begin{example}
With the assumption of Example $1$ we have
\begin{align*}
B^r_t(p)= \left\{\begin{array}{lr} [p,p+t) & t\leq 1 \\ (-\infty, p+t) & t>1  \end{array},\right.
%\end{align*}
%\begin{align*}
B^l_t(p)= \left\{\begin{array}{lr} (p-t,p] & t\leq 1 \\ (p-t, +\infty) & t>1  \end{array},\right.
\end{align*}
 \begin{center}
 and
 \end{center}
 \begin{align*}
 \overline{B^r_t(p)} = \left\{\begin{array}{lr} [p,p+t] & t<1 \\ (-\infty, p+t] & t\geq 1 \end{array},\right.
 %\end{align*}
  %\begin{equation*}
  \overline{B^l_t(p)} = \left \{\begin{array}{lr} [p-t,p] & t<1 \\~
  [ p-t, +\infty ) & t\geq 1  \end{array}.\right.
  \end{align*}
   In the rest of this paper the topology of $X$ is the topology generated by\\ \{${B_t(p), p\in X, t,r\in R}$\}.
  \end{example}
If $T:X\longrightarrow X$ is a continuous map and, $n$ is a natural number we define a new quasi-metric $e_n$ on $X$ by:
 \begin{align*}
  e_n(x,y)=\max_{0\leq i\leq {n-1}}e(T^i(x),T^i(y)).
 \end{align*}
\begin{definition}
 For a natural number $n$, $\epsilon >0$, and a compact subset $K$ of $X$ we say that a subset $F$ of $X$ is an $(n,\epsilon)$-span of $K$ with respect to $T$, if it satisfies the following properties:\\
If $ x\in K$, then there is $y\in F$ such that $e_n(x,y)\leq \epsilon$, and $e_n(y,x)\leq \epsilon$ (i.e.,
%\begin{enumerate}
$x\in \cap_{i=0}^{n-1}T^{-i}\overline{B_\epsilon (T^iy)}).$
%\end{enumerate}
\end{definition}
\begin{definition}
 If $n$ is a natural number, $\epsilon >0$ and $K$ is a compact subset of $X$ then we denote the smallest cardinality of any ($n, \epsilon$)-spanning set of $K$ with respect to $T$ by $r^\prime_n(\epsilon,K) $. (When we need to emphasise on $T$ we shall write $r^\prime
_n(\epsilon, K, T)).$
\end{definition}
Since $K$ is compact, then $r^\prime_n(\epsilon,K)<\infty$. and it is clear that $r^\prime_n(.,K)$ is a non-increasing map on $(0,+\infty)$.\\
If $K$ a compact subset of $X$, and $\epsilon>0$ then $r^\prime(\epsilon,K,T)=\limsup_{n\rightarrow \infty}\frac{\log r^\prime_n(\epsilon,K)}{n}$, we also denote $r^\prime(\epsilon,K,T)$ by $r^\prime(\epsilon,K,T,e)$.\\
Logarithm has essential role in definition of any kind of entropy, because this function is the only function which translate  the physical properties of entropy to mathematical structures of it (for details see Theorem 4.1 of \cite{Wa}).

The value of $r^\prime_n(\epsilon,K)$ could be $\infty$, and $r^\prime_n(.,K)$ is a non-increasing map on $(0,+\infty)$.
\begin{definition}
 Let $h^\prime(T,K)=\lim_{\epsilon\rightarrow 0}r^\prime(\epsilon,K,T)$, where $K$ is a compact subset $X$. Then the topological entropy of $T$ is $h^\prime(T)=\sup_Kh^\prime (T,K)$, where the supremum is taken over the collection of all compact subsets of $X$. We sometimes write $h^\prime_e(T)$ instead of $h^\prime (T)$ to emphasis the dependence on $e$.
 \end{definition}
Now we shall give an equivalent definition. In this definition we use of the idea of separated sets which are dual to spanning sets.\\
Here dual means that in separated notion we pay attention to a set by open balls which are inside of it, but in spanning notion we pay attention to a set by open ball which can be outside of it.\\
 For a natural number $n$, $\epsilon >0$, and a compact subset $K$ of $X$, a subset $E$ of $X$ is called an $(n,\epsilon)$-separated of $K$ with respect to $T$, if it satisfies the following property:\\
If $x,y\in E$, and $x\neq y$, then $e_n(x,y)>\epsilon$ or $e_n(y,x)>\epsilon.$(i.e., If $x,y\in E$, and $x\neq y$, then
%\begin{enumerate}[\bf (1)]
 $y\notin \cap_{i=0}^{n-1}T^{-i}\overline {B^r_\epsilon (T^ix)}$ or
$x\notin \cap_{i=0}^{n-1}T^{-i}\overline {B^r_\epsilon (T^iy)}$).
%\end{enumerate}
 % Obviously that if $e$ is a metric on $X$, then this conditions are equivalent.
 If $n$ is a natural number, $\epsilon >0$ and $K$ is a compact subset of $X$ then $s^\prime_n(\epsilon,K) $ denotes the largest cardinality of any ($n, \epsilon$)-separated subset of $K$ with respect to on $T$. (When we need to emphasis $T$ we shall write $s^\prime
_n(\epsilon, K, T)).$
\begin{theorem}
 $r^\prime_n(\epsilon, K, T)\leq s^\prime_n(\epsilon, K, T)\leq r^\prime_n(\frac{\epsilon}{2}, K, T)$.
  %And hence
  %$s^\prime_n(\epsilon, K, T)<\infty.$
\end{theorem}
 \begin{proof}
 If $E$ is an $(n,\epsilon)$-separated subset of $K$ with the maximal cardinality then $E$ is an $(n,\epsilon)$-spanning set for $K$, because if $x\in K$ then there is $y \in E$ such that $e_n(x,y)\leq \epsilon$ and  $e_n(y,x)\leq \epsilon$. Therefore $r^\prime_n(\epsilon, K, T)\leq s^\prime_n(\epsilon, K, T)$. To show the
 other inequality suppose $E$ is an $(n,\epsilon)$-separated subset of $K$ and $F$ is an $(n,\frac{\epsilon}{2})$-spanning set for $K$. Define $\Phi : E\longrightarrow F$ as the following form:\\
 For given $x\in E$ we choose a point $\Phi (x)\in F$ such that $e_n(x,\Phi (x))\leq \frac{\epsilon}{2}$ and  $e_n(\Phi (x),x)\leq \frac{\epsilon}{2}$.  $\Phi$ is injective, because if $e_n(x,y)\leq \frac{\epsilon}{2}$,  $e_n(y,x)\leq \frac{\epsilon}{2}$,  $e_n(x,z)\leq \frac{\epsilon}{2}$ and  $e_n(z,x)\leq \frac{\epsilon}{2}$ then $e_n(y,z)\leq \epsilon$ and $e_n(z,y)\leq \epsilon$, which is a contradiction. Therefore
 the cardinality of $E$ is not greater than the cardinality of $F$. Hence $s^\prime_n(\epsilon, K, T)\leq r^\prime_n(\frac{\epsilon}{2}, K, T)$. \\
Since $K$ is compact, then $s^\prime_n(\epsilon,K)<\infty$, and it is obvious that $s_n(.,K)$ is a non-increasing map on $(0,+\infty)$.
\end{proof}
Theorem $2.1$ also implies that $s^\prime_n(\epsilon,K,T)<\infty$.

  If $K$ is a compact subset of $X$, $\epsilon>0$ then we define $s^\prime(\epsilon,K,T)$ by
   $\limsup_{n\rightarrow \infty}\frac{\log s^\prime_n(\epsilon,K)}{n}.$\\ We also write $s^\prime(\epsilon,K,T,e)$ if we need to emphasis to the quasi-metric $e$. The value of $s^\prime_n(\epsilon,K)$ could be $\infty$, and it is obvious that $s_n(.,K)$ is a non-increasing map of $(0,+\infty)$.\\
 {\bf Remark:}\\
 We have\\
 1) $r^\prime(\epsilon, K, T)\leq s^\prime(\epsilon, K, T)\leq r^\prime(\frac{\epsilon}{2}, K, T)$,\\
 2) $s(.,K)$ is non-increasing map of $(0,+\infty)$,\\
 3) $h^\prime(T,K)=\lim_{\epsilon\rightarrow 0}s^\prime(\epsilon,K,T)$, $h^\prime(T)=\sup_Kh^\prime (T,K)= \lim_{\epsilon\rightarrow 0}s^\prime(\epsilon,K,T)$, where the supremum is taken over the collection of all compact subset of $X$.
  \section{The entropy of maps on metric spaces and the entropy of maps on quasi-metric spaces }
  \label{sec3}
  In this section we define a metric $d_e$ by quasi-metric $e$, and we study the relation between topological entropy $h_{d_e(T)}$ and $h_e^\prime(T),$ where $h_{d_e}(T)$ is the topological entropy $h(T)$ with respect to the metric $d_e$ and $h_e^\prime(T)$ is the topological entropy $h^\prime (T)$ with respect to the quasi-metric $e$.\\
  If $(X,e)$ is a quasi-metric space then it is obvious that $(X,d_e)$ is a metric space, where $d_e$ is a metric on $X$ defined by:
    \begin{align*}
   & d_e:X\times X\longrightarrow [0,+\infty)\\
   & d_e(x,y)=\frac{e(x,y)+e(y,x)}{2}.
     \end{align*}
   For the metric space $(X,d_e)$, suppose the map $T:(X,d_e)\longrightarrow (X,d_e)$ is continuous and $h_{de}(T)$ is its topological entropy, (i.e.,
     \begin{align*}
      h_{de}(T)&=\sup_K\lim_{\epsilon \rightarrow 0}\limsup_{n\rightarrow \infty}\frac{\log (r_n(\epsilon, K, T, d_e))}{n}\\&=\sup_K\lim_{\epsilon \rightarrow 0}\limsup_{n\rightarrow \infty}\frac{\log (s_n(\epsilon, K, T, d_e))}{n},
        \end{align*}
         where the supremum is taken over the collection of all compact subsets of $X$ and \\ $r_n(\epsilon, K, T, d_e)$ denotes the smallest cardinality of any $(n,\epsilon)$-spanning set
   for $K$ with respect to the metric $d_e$ and $T$, and $s_n(\epsilon, K, T, d_e)$ denotes the largest cardinality of any $(n,\epsilon)$-separated subset of $K$ with respect to the metric $d_e$ and $T$).\\
    Since $e(x,y)<\epsilon $ and $e(y,x)<\epsilon$ then $d_e(x,y)<\epsilon,$ so $B(x,e,\epsilon)\subseteq B(x,d_e,\epsilon).$ Since $d_e(x,y)<\epsilon,$ then $e(x,y)<2\epsilon $ and $e(y,x)<2\epsilon,$ so $B(x,d_e,\epsilon)\subseteq B(x,e,2\epsilon).$ Therefore the topology generated by open sets $\{B_{d_e}(x,\epsilon), x\in X, \epsilon>0\}$ is the same as the topology generated by open sets $\{B_r(x), x\in X, r>0\}.$\\
  \begin{theorem}
   $h_{d_e}(T)=h_e^\prime (T).$
   \end{theorem}
   \begin{proof}
   We have
 % \begin{align*}
 \begin{center}
  $d_{en}(x,y)=\max \{d_e(x,y),...,d_e(T^{n-1}x,T^{n-1}y)\}=\max\{\frac{e(x,y)+e(y,x)}{2},...,\frac{e(T^{n-1}x,T^{n-1}y)+e(T^{n-1}y,T^{n-1}x)}{2}\}.$
  %\end{align*}
  \end{center}
  So if $E$ is an $(n,\epsilon)$-spanning set for $K$ of minimal cardinality with respect to the quasi-metric $e$ and $T$, then $E$ is an $(n,\epsilon)$-spanning set for $K$ with respect to the metric $d_e$ and $T$, because if $e(T^{i-1}x,T^{i-1}y)\leq \epsilon$ and $e(T^{i-1}y,T^{i-1}x)\leq \epsilon$ for $i=0,...,n-1$, then $d_{en}(x,y)\leq \epsilon$. Hence $r_n(\epsilon,K,T,d_e)\leq r^\prime_n(\epsilon,K,T,e).$ If $E$ is an $(n,\epsilon)$-separated subset for $K$ of maximal cardinality with respect to the quasi-metric $d_e$ and $T$, then $E$ is an $(n,\epsilon)$-separated subset for $K$ with respect to the metric $e$ and $T$, therefore $s_n(\epsilon,K,T,d_e)\leq s^\prime_n(\epsilon,K,T,e).$ If $E$ is an $(n,\epsilon)$-spanning set for $K$ of minimal cardinality with respect to the metric $d_e$ and $T$, then $E$ is an $(n,2\epsilon)$-spanning set for $K$ with respect to the quasi-metric $e$ and $T$, because if $d_{en}(x,y)\leq \epsilon$, then $\frac{e(T^{i-1}x,T^{i-1}y)+e(T^{i-1}y,T^{i-1}x)}{2}\leq \epsilon$ for $i=0,...,n-1$. Thus $e(T^{i-1}x,T^{i-1}y)\leq 2\epsilon $ and
 $e(T^{i-1}y,T^{i-1}x)\leq 2\epsilon $, for $i=0,...,n-1$. So $r^\prime_n(2\epsilon,K,T,e)\leq r_n(\epsilon,K,T,d_e).$ Hence  $r^\prime_n(2\epsilon,K,T,e)\leq r_n(\epsilon,K,T,d_e)\leq r^\prime_n(\epsilon,K,T,e)$. Thus $h_{d_e}(T)=h_e^\prime (T).$
 \end{proof}
 Now we define another metric on $(X,e).$
 If $(X,e)$ is a quasi-metric space, then $(X,m_e)$ is a metric space, where $m_e(x,y)=\max \{e(x,y), e(y,x)\}$.\\
  If $e(x,y)<\epsilon $ and $e(y,x)<\epsilon$ then $m_e(x,y)<\epsilon,$ so $B(x,e,\epsilon)\subseteq B(x,m_e,\epsilon).$ If $m_e(x,y)<\epsilon,$ then $e(x,y)<\epsilon $ and $e(y,x)<\epsilon,$ so $B(x,m_e,\epsilon)\subseteq B(x,e,\epsilon).$ Therefore the topology generated by open sets $\{B_{m_e}(x,\epsilon), x\in X, \epsilon>0\}$ is the same as the topology generated by open sets $\{B_r(x), x\in X, r>0\}.$
   \begin{theorem}
    $h_{m_e}(T)=h_e^{\prime} (T).$
   \end{theorem}
   \begin{proof}
   Since
   \begin{center}
   $m_{en}(x,y)=\max \{m_e(x,y),...,m_e(T^{n-1}x,T^{n-1}y)\} =\max \{\max\{e(x,y),e(y,x)\},..., \max\{e(T^{n-1}(x),T^{n-1}(y)), e(T^{n-1}(y),T^{n-1}(x))\}\},$
      \end{center}
  then $m_{en}(x,y)<\epsilon$. Hence\\
  %\begin{center}
   $\max \{\max\{e(x,y),e(y,x)\},...,\max\{e(T^{n-1}(x),T^{n-1}(y)), e(T^{n-1}(y),T^{n-1}(x))\}\}$\\ $<\epsilon$.
 %\end{center}
 Therefore $\max\{e(T^{i}(x),T^{i}(y)), e(T^{i}(y),T^{i}(x))\}<\epsilon$, for every $1\leq i\leq n-1$. So $e(T^{i}(x),T^{i}(y))<\epsilon$ and $e(T^{i}(y),T^{i}(x))<\epsilon$, $1\leq i\leq n-1$.\\ Hence $e_n(x,y)<\epsilon$ and $e_n(y,x)<\epsilon$. This proves that any $(n,\epsilon)$-spanning set for $K$ with respect to the metric $m_e$ is an $(n,\epsilon)$-spanning set for $K$ with respect to the quasi-metric $e$.\\
  If  $e_n(x,y)<\epsilon$ and $e_n(y,x)<\epsilon$ then $e(T^{i}(x),T^{i}(y))<\epsilon$ and $e(T^{i}(y),T^{i}(x))<\epsilon$, $1\leq i\leq n-1$. \\So  \\$\max \{\max\{e(x,y),e(y,x)\},...,\max\{e(T^{n-1}(x),T^{n-1}(y)), e(T^{n-1}(y),T^{n-1}(x))\}\}\\<\epsilon$.\\ Therefore $m_{en}(x,y)<\epsilon$. Thus any $(n,\epsilon)$-spanning set for $K$ with respect to the quasi-metric $e$ is an $(n,\epsilon)$-spanning set for $K$ with respect to the metric $m_e$.\\
  The set $E$ is an $(n,\epsilon)$-spanning set for $K$ with respect to the metric $m_e$ if and only if $E$ is an $(n,\epsilon)$-spanning set for $K$ with respect to the quasi-metric $e$. Hence $h_{m_e}(T)=h_e^{\prime} (T).$
  \end{proof}
  Theorems $3.1$ and $3.2$ imply to the following corollary.\\
 {\bf Corollary:} $h_{m_e}(T)=h_{d_e}(T).$
 \section{ Another approach to the topological entropy}
 \label{sec4}
A subset $E$ of $X$ is called a $(\epsilon,n)$-span of $K$ if for given $ x\in K$, there is $y\in F$ with $e_n(x,y)\leq \epsilon$, or $e_n(y,x)\leq \epsilon$. (i.e.
 %\begin{enumerate}[\bf (1)]
$y\in \cap_{i=0}^{n-1}T^{-i}\overline{B^r_\epsilon (T^ix)},$
 or
$x\in \cap_{i=0}^{n-1}T^{-i}\overline{B^r_\epsilon (T^iy)}).$\\
% \end{enumerate}
  A subset $E$ of $K$ is called $(\epsilon, n)$-separated with respect to $T$ if $x,y\in E$, and $x\neq y$, then $e_n(x,y)>\epsilon$ and $e_n(y,x)>\epsilon.$(i.e., If $x,y\in E$, and $x\neq y$, then $y\notin \cap_{i=0}^{n-1}T^{-i}\overline {B^r_\epsilon (T^ix)}$ and $x\notin \cap_{i=0}^{n-1}T^{-i}\overline {B^r_\epsilon (T^iy)}$).\\
 %Also we replaced $r^\prime_n(\epsilon, K, T)$ and $s^\prime_n(\epsilon, K, T)$ by $r^{\prime\prime}_n(\epsilon, K, T)$ and $s^{\prime\prime}_n(\epsilon, K, T)$ respectively, where
 $r^{\prime\prime}_n(\epsilon, K, T)$ denotes the smallest cardinality of any ($\epsilon,n$)-spanning set for $K$ with respect to $T$ and $s^{\prime\prime}_n(\epsilon, K, T)$ denotes the largest cardinality of any ($\epsilon,n$)-separated subset of $K$ with respect to $T$.\\
  {\bf Remarks:}\\
  We have\\
  1) $r^{\prime\prime }(\epsilon, K, T)\leq s^{\prime\prime}(\epsilon, K, T)\leq r^{\prime\prime}(\frac{\epsilon}{2}, K, T)$,\\
  2) $s(.,K)$ is non-increasing map on $(0,+\infty)$,\\
  3) $h^{\prime \prime}(T,K)=\lim_{\epsilon\rightarrow 0}s^{\prime \prime}(\epsilon,K,T)$, so $h^{\prime\prime}(T)=\sup_Kh^{\prime\prime} (T,K)$\\ $=\lim_{\epsilon\rightarrow 0}s^{\prime \prime}(\epsilon,K,T)$, \\where the supremum is taken over the collection of all compact subset of $X$.\\
  We define:
  \begin{align*}
  h^{\prime\prime}_{e}(T)&=\sup_K\lim_{\epsilon \rightarrow 0}\limsup_{n\rightarrow \infty}\frac{\log (r^{\prime\prime}_n(\epsilon, K, T, e))}{n}\\ &=\sup_K\lim_{\epsilon \rightarrow 0}\limsup_{n\rightarrow \infty}\frac{\log (s^{\prime\prime}_n(\epsilon, K, T, e))}{n}.
    \end{align*}

\begin{theorem}
$ h^{\prime\prime}_{e}(T) \leq h^{\prime}_{e}(T).$
\end{theorem}
\begin{proof}
The inequalities
 %\begin{align*}
$ r^{\prime\prime}_n(\epsilon, K, T, e) \leq r^{\prime}_n(\epsilon, K, T, e)$
 %\end{align*}
 and
% \begin{align*}
 $s^{\prime\prime}_n(\epsilon, K, T, e)\leq s^{\prime}_n(\epsilon, K, T, e),$\\
%\end{align*}
imply that
 \begin{align*}
  h^{\prime\prime}_{e}(T)&=\sup_K\lim_{\epsilon \rightarrow 0}\limsup_{n\rightarrow \infty}\frac{\log (r^{\prime\prime}_n(\epsilon, K, T, e))}{n}\\ &=\sup_K\lim_{\epsilon \rightarrow 0}\limsup_{n\rightarrow \infty}\frac{\log (s^{\prime\prime}_n(\epsilon, K, T, e))}{n}\\&\leq \sup_K\lim_{\epsilon \rightarrow 0}\limsup_{n\rightarrow \infty}\frac{\log (r^{\prime}_n(\epsilon, K, T, e))}{n}\\ &=\sup_K\lim_{\epsilon \rightarrow 0}\limsup_{n\rightarrow \infty}\frac{\log (s^{\prime}_n(\epsilon, K, T, e))}{n}\\&= h^{\prime}_{e}(T).
    \end{align*}
\end{proof}
The following example shows that a $(\epsilon,n)$-span may not be the same as an $(n,\epsilon)$-span.
\begin{example}
Let quasi-metric $e$ be as Example $1$, $K=[0,1]$ and $T=I$. Then a $(\epsilon,n)$-span of $K$ with respect to $T$ is\\
\begin{align*}
  F:=  \left\{\begin{array}{lr} x & \epsilon \geq 1 \\ \{\frac{\epsilon}{2},\frac{2\epsilon}{2},\frac{3\epsilon}{2},...,([\frac{2}{\epsilon}]+1)\epsilon \}& \epsilon<1  \end{array},\right.
  \end{align*}
 where $x$ is any point in $K$.\\
 And
 \begin{align*}
  E:=  \left\{\begin{array}{lr} \emptyset & \epsilon \geq 1 \\ \{0, x \}& \epsilon<1  \end{array}.\right.
 \end{align*}
 is a $(\epsilon,n)$-separated set of $K$ with respect to $T$,
 where $\epsilon<x<1$.
 \end{example}
 Let $(X,e)$ be a quasi-metric space. A map $T:X\rightarrow X$ is called uniformly continuous if for given $\epsilon>0$ there is $\delta >0$ such that $e(x,y)<\delta$ implies $e(T(x),T(y))<\epsilon.$ The space of all uniformly continuous maps of a quasi-metric space $(X,e)$ is denoted by $UC(X,e)$.\\
 \begin{theorem}
If $(X,e)$ is a quasi-metric space, $T\in UC(X,e)$ and $m$ is a natural number then $h^{\prime\prime}(T^m)=mh^{\prime\prime}(T).$
 \end{theorem}
 \begin{proof}
 Let $F$ be a $(\epsilon,mn)$-span $K$ with respect $T$. So for  each $x\in K$ there is $y\in F$ such that
 \begin{center}
 $\max_{1\leq i\leq mn-1}\{e(T^i(x),T^i(y)\}<\epsilon$ or $\max_{1\leq i\leq mn-1}\{e(T^i(y),T^i(x)\}<\epsilon.$
 \end{center}
  Hence for each $x\in K$ there is $y\in F$ such that
\begin{center}
 $\max_{1\leq i\leq n-1}\{e(T^{mi}(x),T^{mi}(y)\}<\epsilon$ or $\max_{1\leq i\leq n-1}\{e(T^{mi}(y),T^{mi}(x)\}<\epsilon$.
 \end{center}
 Thus $F$ is a $(\epsilon,n)$-span of $K$ with respect $T^m$. Hence
 \begin{center}
 $r^{\prime\prime}_n(\epsilon,K,T^m)\leq r^{\prime\prime}_{mn}(\epsilon,K,T),$
 \end{center}
 and we have
 \begin{center}
 $\frac{1}{n}\log r^{\prime\prime}_n(\epsilon,K,T^m)\leq \frac{m}{mn}\log r^{\prime\prime}_{mn}(\epsilon,K,T),$
 \end{center}
so $h^{\prime\prime}(T^m)\leq mh^{\prime\prime}(T).$\\
 If $T\in UC(X,e)$ then for given $\epsilon >0$ there is $\delta>0$ such that
 \begin{center}
 $e(x,y)<\delta \Rightarrow e_m(x,y)<\epsilon$ and $e(y,x)<\delta \Rightarrow e_m(y,x)<\epsilon.$
 \end{center}
 $e(x,y)<\delta$ or  $e(y,x)<\delta$ imply that $e_m(x,y)<\epsilon$ or $e_m(y,x)<\epsilon$. Therefore a $(\delta,n)$-spanning set of $K$ with respect to $T^m$ is an $(\epsilon,mn)$-spanning set of $K$ with respect to $T$. Thus
  \begin{center}
  $r^{\prime\prime}_n(\delta,K,T^m)\geq r^{\prime\prime}_{mn}(\epsilon,K,T).$
  \end{center}
  Hence
  \begin{center}
  $r^{\prime\prime}(\delta,K,T^m)\geq mr^{\prime\prime}(\epsilon,K,T).$
 \end{center}
 So $h^{\prime\prime}(T^m)\geq mh^{\prime\prime}(T).$ Thus $h^{\prime\prime}(T^m)=mh^{\prime\prime}(T).$
 \end{proof}
 \section{Conclusion}
  \label{sec5}
  In this work we present  three methods to consider topological entropy. We show that if we add symmetric  property on the space via quasi-metric, then the entropy of a continuous map do not decreases! (Theorem 4.1). Consideration of the validity of the converse of   Theorem 4.1 is a topic for further research.
  
  \section*{Acknowledgments}
The authors would like to thank  the referees for their valuable comments.

%%%% Bibliography  %%%%%%%%%%

\end{document}